\documentclass[12pt,a4paper,oneside]{amsart}

\usepackage{anysize}
\marginsize{2cm}{2cm}{1.5cm}{1.5cm}
\newtheorem{thm}{Theorem}[section]
\newtheorem{prp}[thm]{Proposition}
\newtheorem{lem}[thm]{Lemma}
\newtheorem{cor}[thm]{Corollary}

\theoremstyle{definition}
\newtheorem{dfn}{Definition}[section]
\newtheorem{rem}[dfn]{Remark}

\newcommand{\st}{:\;}

\def\R{{\mathbb R}}%
\def\N{{\mathbb N}}%
\renewcommand{\phi}{\varphi}
\def\e{\varepsilon}

\providecommand{\parenth}[1]{\left(#1\right)}%
\providecommand{\braces}[1]{\left\{#1\right\}}%
\newcommand{\iprod}[2]{\left\langle#1,#2\right\rangle}%

\newcommand{\ball}[1]{\mathbf{B}^{#1}}
\newcommand{\norm}[1]{\left\|#1\right\|}
\newcommand{\enorm}[1]{\left|#1\right|}
\providecommand{\card}[1]{\lvert#1\rvert}%
\newcommand{\dist}[2]{\operatorname{dist}\!\left( #1, #2 \right)}%

\usepackage{xcolor}
\usepackage{hyperref}
\hypersetup{
	colorlinks   = true, 
	urlcolor     = blue, 
	linkcolor    = blue, 
	citecolor    = red 
}
\newcommand{\Href}[2]{\hyperref[#2]{#1~\ref{#2}}}
\DeclareMathOperator{\mconame}{\delta}
\DeclareMathOperator{\mglname}{\varrho}
\newcommand{\rmglx}[1]{\rho^{-1}_{X}\!\!\!\:\left( #1 \right)}%
\newcommand{\mcox}[1]{\mconame\nolimits_{X}\!\left( #1 \right)}%
\newcommand{\mglx}[1]{\mglname\nolimits_{X}\!\!\!\:\left( #1 \right)}%
\newcommand{\zetap}[1]{\zeta^{+}_{X}\!\left( #1 \right)}%
\newcommand{\zetam}[1]{\zeta^{-}_{X}\!\left( #1 \right)}%

\title{No-dimensional Helly's theorem in uniformly convex Banach spaces}

\author{Grigory Ivanov\address{Grigory Ivanov: 
Pontifícia Universidade Cat\'olica do Rio de Janeiro \\
Departamento de Matematica,
Rua Marquês de São Vicente, 225\\
Edif{\'i}cio Cardeal Leme, sala 862,
22451-900 G{\'a}vea, Rio de Janeiro, Brazil}
\email{grimivanov@gmail.com}}
\thanks{The author is supported by Projeto Paz and Coordenacao de Aperfeicoamento de Pessoal de Nivel Superior - Brasil (CAPES) - 23038.015548/2016-06}

\date{\today}

\begin{document}
\begin{abstract}
We study the ``no-dimensional'' analogue of Helly's theorem in Banach spaces. Specifically, we obtain the following no-dimensional Helly-type results for uniformly convex Banach spaces: Helly's theorem, fractional Helly's theorem, colorful Helly's theorem, and colorful fractional Helly's theorem. 

The combinatorial part of the proofs for these Helly-type results is identical to the Euclidean case as presented in \cite{adiprasito2020theorems}. The primary difference lies in the use of a certain geometric inequality in place of the Pythagorean theorem. This inequality can be explicitly expressed in terms of the modulus of convexity of a Banach space.
\end{abstract}

\maketitle
\section{Introduction}
Helly's theorem \cite{helly1923mengen} is a classical result in convex geometry. It states that if the intersection of any subfamily of a finite family of convex sets in $\R^d$, consisting of at most $d + 1$ sets, contains a common point, then all the sets in the family share at least one common point. This fundamental result has seen numerous generalizations and extensions; for an overview, we refer to the recent survey \cite{barany2022helly}. In this note, we study a no-dimensional version of Helly's theorem in Banach spaces, motivated by the groundbreaking paper \cite{adiprasito2020theorems}, where the following no-dimensional version of Helly's theorem in the Euclidean case was introduced:

Throughout this paper, $[n]$ represents the set $\{1, \dots, n\}$ for some natural number $n$, and $\dist{a}{S}$ denotes the distance between a point $a$ and a set $S$ in the underlying normed space. We use $\ball{}$ to denote the unit ball centered at the origin in the underlying normed space. The size of a family of sets $\mathcal{F}$ is denoted as $\card{\mathcal{F}}$.

\begin{prp}[No-dimensional Euclidean Helly's theorem] \label{prp:no-dim_Helly}
Fix $k \in [n]$, and set $r_k\!\parenth{\R^d} = \frac{1}{\sqrt{k}}$. The sequence $r_k\!\parenth{\R^d}$ satisfies the following property: if $K_1,\ldots,K_n$ are convex sets in $\R^d$ such that the Euclidean unit ball $\ball{}$ intersects $\bigcap\limits_{j \in J}K_j$ for every $J\subset [n]$ with $\card{J} = k$, then there exists a point $x \in \R^d$ such that
\[
\dist{x}{K_i} \leq r_k \! \parenth{\R^d} \quad \text{for all} \quad i \in [n].
\]
\end{prp}

The particular dimension $d$ is not essential in this result; one could consider the case of an infinite-dimensional Hilbert space from the outset.

Our main result is an extension of this theorem to the case of uniformly convex Banach spaces, where we bound the distance between a point and the sets in terms of the modulus of convexity of the Banach space. For clarity, we state our main result here for the case of a super-reflexive Banach space. The interrelations between super-reflexive and uniformly convex Banach spaces will be discussed in \Href{Section}{sec:moduli_Banach_spaces}.

The following is the main result of the paper.

\begin{thm}[No-dimensional Helly's theorem]\label{thm:no-dim_Helly_Banach}
Fix $k \in [n]$. Let $X$ be a super-reflexive Banach space. There exist a strictly positive constant $C_r$ and $w \in (0, 1/2]$ such that the sequence $r_k(X) = C_r k^{-w}$ satisfies the following property:

If $K_1,\ldots,K_n$ are convex sets in $X$ such that the unit ball $\ball{}$ intersects $\bigcap\limits_{j \in J}K_j$ for every $J\subset [n]$ with $\card{J} = k$, then there exists a point $x \in X$ such that
\[
\dist{x}{K_i} \leq r_k (X) \quad \text{for all} \quad i \in [n].
\]
\end{thm}

We will provide explicit bounds on $r_k(X)$ using the modulus of convexity, a function associated with any Banach space that ``measures'' the convexity of the unit ball. Furthermore, our method yields the estimate $r_k (\R^d) = \frac{4}{\sqrt{k}}$ for \Href{Proposition}{prp:no-dim_Helly} in the Euclidean case.

In \Href{Section}{sec:proofs_no-dim_Helly}, we will formulate and prove several no-dimensional Helly-type results for uniformly convex Banach spaces: fractional Helly's theorem, colorful Helly's theorem, and colorful fractional Helly's theorem. Here, we present a version of the colorful fractional Helly's theorem for a super-reflexive space:

\begin{thm}
\label{thm:no-dim_fractional_colorful_Helly}
Let $\alpha \in (0,1]$ and $\beta = 1-(1-\alpha)^{\frac{1}{k}}$. Let $X$ be a super-reflexive Banach space. There exist a strictly positive constant $C_r$ and $w \in (0, 1/2]$ such that the sequence $r_k(X) = C_r k^{-w}$ satisfies the following property:

If, for an $\alpha$ fraction of ``rainbow'' $k$-tuples $K_1 \in \mathcal{F}_1, \dots, K_k \in \mathcal{F}_k$ of sets in finite, non-empty families of convex sets $\mathcal{F}_1, \dots, \mathcal{F}_k$ in $X$, the set $\bigcap\limits_{i \in [k]} K_i$ contains a point in the unit ball $\ball{}$, then there exist $x \in X$ and $i \in [k]$ such that at least $\beta |\mathcal{F}_i|$ elements of $\mathcal{F}_i$ intersect the ball $r_k(X) \ball{} + x$.
\end{thm}

Following \cite{adiprasito2020theorems}, we also formulate a standard corollary of Helly's theorem adapted to our setting -- a no-dimensional centerpoint theorem:

\begin{thm}[No-dimensional centerpoint theorem]
\label{thm:no-dim_centerpoint}
Let $X$ be a super-reflexive Banach space. There exist a strictly positive constant $C_r$ and $w \in (0, 1/2]$ such that the sequence $r_k(X) = C_r k^{-w}$ satisfies the following property: for any $n$-point subset $P$ of $\ball{}$ and any $k$, there exists a point $x$ such that any half-space containing the ball $r_k(X) \ball{} + x$ contains at least $\frac{n}{k}$ points of $P$.
\end{thm}

We observed that the original proof strategy of \Href{Proposition}{prp:no-dim_Helly} and other Helly-type results in the Euclidean case from \cite{adiprasito2020theorems} works perfectly in uniformly convex Banach spaces. One only needs to replace a certain Euclidean relation with a relation involving the modulus of convexity. For this reason, we will follow the original proofs of Helly-type results from \cite{adiprasito2020theorems}, substituting the use of the Pythagorean theorem with certain inequalities. Our key insight is as follows:

\begin{lem}
\label{lem:deviation_from_supp_hyperplane}
Fix $\rho > 0$. Let $H$ be a supporting hyperplane at point $x$ to the unit ball $\ball{}$ of a normed space $\parenth{X, \norm{\cdot}}$. Let $H^{-}$ and $H^{+}$ denote the closed half-spaces with boundary $H$ that contain $\ball{}$ and do not contain $\ball{}$, respectively. Then
\[
\inf\limits_{
y \in H^{+}  \st \norm{x-y} \geq \rho} \norm{y}=
\inf\limits_{
y \in H  \st \norm{x-y} = \rho} \norm{y} 
\quad \text{and} \quad
\sup\limits_{
z \in H^{-}  \st \norm{x-z} \leq \rho} \norm{z}=
\sup\limits_{
z \in H  \st \norm{x-z} = \rho} \norm{z}.
\]
\end{lem}

It is a natural assumption for a specialist in geometric Banach space theory that the rightmost infimum can be bounded from  below by the modulus of convexity, while the supremum can be bounded from above by the so-called modulus of smoothness of a Banach space. We will demonstrate that this is indeed the case.

The author believes that the geometric inequality underlying the proofs of \Href{Theorem}{thm:no-dim_Helly_Banach} and \Href{Theorem}{thm:no-dim_fractional_colorful_Helly} is more intriguing than the theorems themselves, as it reveals a certain duality. It is well known that Helly's theorem is dual to Carathéodory's theorem, which states that every point in the convex hull  of a set $S \subset \R^d$ is a convex combination of at most
$d+1$ points of $S$. A similar duality emerges in the case of no-dimensional results, and we will now elaborate on this.

The author previously provided a geometric proof of a no-dimensional Carathéodory theorem in uniformly smooth Banach spaces \cite{ivanov2021approximate}. Generally, a no-dimensional Carathéodory theorem requires properties related to the smoothness of the space (see the celebrated Maurey lemma \cite{pisier1980remarques}). Interestingly, a Banach space is uniformly smooth if and only if its dual is uniformly convex (for more details and definitions, see \cite{DiestelEng}). Moreover, when examining the proofs in detail, the geometric inequalities underlying the no-dimensional Helly and Carathéodory theorems appear to be somewhat dual to each other. As we will demonstrate in \Href{Section}{sec:deviation_from_hyperplane}, these inequalities are direct corollaries of the identities in \Href{Lemma}{lem:deviation_from_supp_hyperplane}.

The author is also interested in exploring a probabilistic proof of the no-dimensional Helly's theorem. The probabilistic proof of the no-dimensional Carathéodory theorem yields several corollaries, such as the no-dimensional versions of the selection lemma and the weak $\epsilon$-net theorem in a Banach space of type $p > 1$, as shown in \cite{ivanov2019no}.

\medskip
\noindent {\textbf{Open problem:}}  
Is there a version of the no-dimensional Helly's theorem and a corresponding dual version of Maurey's lemma in a Banach space of cotype $q \geq 2$?

\section{Geometrical idea behind the proof}

We begin with a sketch of the proof of \Href{Proposition}{prp:no-dim_Helly}. Our goal is to demonstrate that \Href{Lemma}{lem:deviation_from_supp_hyperplane} is the key non-combinatorial insight in the proof, guiding the estimates needed for the case of Banach spaces.

\begin{proof}[Sketch of the proof of \Href{Proposition}{prp:no-dim_Helly}]
For contradiction, assume that for any point $p \in \R^d$, there exists a set $K \in \mathcal{F}$ such that $\dist{p}{K} > r_k(\R^d)$. 

We will inductively construct a sequence of sets $\{K_i\}$ from $\mathcal{F}$ such that the intersections $\bigcap_{i \in [j]} K_i$ progressively move farther away from the origin. Quantitatively, we want to choose the sequence $\{K_i\}$ such that 
\begin{equation}
\label{eq:euclidean_deviation_no_dim_Helly}
 \dist{0}{\bigcap_{i \in [j]} K_i} > \sqrt{j} \cdot r_k(\R^d).
\end{equation}
This will solve the problem because, when $j = k$, the inequality implies that there are $k$ sets in the family whose intersection does not intersect the unit ball, since $\sqrt{k} \cdot r_k(\R^d) = 1$.

The inductive construction is straightforward and follows from the following simple Euclidean corollary of \Href{Lemma}{lem:deviation_from_supp_hyperplane}. 

\begin{prp}
\label{prp:euclidean_min_deviation}
Assume that a point $p$ and two convex sets $K$ and $L$ in $\R^d$ satisfy the following conditions:
\begin{enumerate}
\item $\dist{0}{K} = \norm{p} > \rho_1$; 
\item $\dist{p}{L}  > \rho_2$.
\end{enumerate}
Then $\dist{0}{K \cap L} > \sqrt{\rho_1^2 + \rho_2^2}$, with the convention that $\dist{q}{\emptyset} = \infty$.
\end{prp}

Assume we have found sets $K_1, \dots, K_j$ satisfying \eqref{eq:euclidean_deviation_no_dim_Helly}. Let $p$ be the metric projection of the origin onto $\bigcap\limits_{i \in [j]} K_i$. By our assumption, there exists a set $K_{j+1} \in \mathcal{F}$ such that $\dist{p}{K_{j+1}} > r_k(\R^d)$. By \Href{Proposition}{prp:euclidean_min_deviation}, we obtain the desired inequality:
\[
 \dist{0}{\bigcap_{i \in [j+1]} K_i} > \sqrt{j+1} \cdot r_k(\R^d).
\]
\end{proof}

\Href{Proposition}{prp:euclidean_min_deviation} is a direct corollary of the law of cosines. In turn, \Href{Lemma}{lem:deviation_from_supp_hyperplane} reduces it to the Pythagorean identity. What we will need to prove our Helly-type results is a version of the inequality in \Href{Proposition}{prp:euclidean_min_deviation} for normed spaces. Since \Href{Lemma}{lem:deviation_from_supp_hyperplane} works in any normed space, it reduces the problem to finding the minimal deviation of a point on a supporting hyperplane to the unit ball from the origin. 

This is where the modulus of convexity, defined in the next section, becomes essential as it measures this deviation. Moreover, we will see that the deviation itself can serve as an alternative ``modulus of convexity,'' equivalent to the original one. 

For the Helly-type results, all algorithms and proofs from \cite{adiprasito2020theorems} apply to our case with the necessary adaptation of using an appropriate version of \Href{Proposition}{prp:euclidean_min_deviation}. Some routine analysis is required to bound the sequence $\{r_k(X)\}$ in the general case.

\section{Uniformly convex Banach spaces and their properties}
\label{sec:moduli_Banach_spaces}

In this section, we recall some definitions from geometric Banach space theory. We refer the reader to the excellent books \cite{DiestelEng} and \cite{Lindenstrauss1979} for comprehensive overviews of results on the modulus of convexity and related geometric concepts.

For a normed space $\parenth{X, \norm{\cdot}}$, we define the modulus of convexity as follows:
\begin{equation*}
\mcox{\e}: = \inf \left\{ 1 - \frac{\norm{x+y}}{2} \ :\ x,y \in \ball{},\ \norm{x - y} \geq \e \right\}.
\end{equation*}
The function $\mcox{\cdot}: [0, 2] \to [0, 1]$ is called the \emph{modulus of convexity} of $X$. A Banach space is said to be \emph{uniformly convex} if its modulus of convexity is strictly positive on $(0, 2]$.

We say that the modulus of convexity $\mcox{\cdot}$ of $X$ is of \emph{power type} $q$ if it satisfies the inequality:
\begin{equation}
\label{eq:mcox_power_type}
\mcox{\e} \geq C_X \e^{q} \quad \text{for some} \quad C_X > 0,\ q \geq 2,\ \e \in [0, 1].
\end{equation}
It is well known that an $L_p$ space for $1 < p < \infty$ is uniformly convex, and its modulus of convexity is well understood; in particular, it satisfies inequality \eqref{eq:mcox_power_type} for certain parameters.

Uniform convexity of the norm is not stable under small perturbations. However, an equivalent renormalization of the space only introduces a constant factor to $r_k(X)$. Thus, to prove \Href{Theorem}{thm:no-dim_Helly_Banach}, it suffices to establish the result for some equivalent norm on the space. Classical results from Banach space theory, due to Enflo \cite{Enflo1972} and Pisier \cite{pisier1975martingales}, allow us to consider only uniformly convex spaces with a modulus of convexity of power type. Summarizing their results, we state the following proposition:

\begin{prp} \label{prp:superreflexivity_uniform_conv}
The following assertions are equivalent:
\begin{itemize}
\item A Banach space $X$ admits an equivalent uniformly convex norm;
\item A Banach space $X$ admits an equivalent uniformly convex norm with a modulus of convexity of power type $q$;
\item A Banach space $X$ is super-reflexive. 
\end{itemize}
\end{prp}

However, we believe that yet another ``modulus of convexity'' is more convenient for our purposes, particularly when estimating the deviation of a point in a supporting half-space of the unit ball from the ball itself. Let us introduce some definitions following \cite{ivanov2017new} and \cite{Ivanov_supp_modulus}.

We say that $y$ is \emph{quasi-orthogonal} to a non-zero vector $x \in X$, denoted by $y \urcorner x$, if there exists a functional $p$ that attains its norm on $x$ and vanishes on $y$. The following conditions are equivalent:
\begin{itemize}
\item $y$ is quasi-orthogonal to $x$;
\item For any $\lambda \in \R$, the vector $x + \lambda y$ lies in the supporting hyperplane to the ball $\norm{x}\ball{}$ at $x$;
\item For any $\lambda \in \R$, the inequality $\norm{x + \lambda y} \geq \norm{x}$ holds;
\item $x$ is orthogonal to $y$ in the sense of Birkhoff--James (see \cite{DiestelEng}, Ch. 2, \S 1, and \cite{AlonsoMartiniWu_birkhoff_orthogonality}).
\end{itemize}

In the Euclidean case, quasi-orthogonality is equivalent to standard orthogonality.

We associate with any normed space $X$ the function $\zetam{\cdot} \colon [0, \infty) \to  [0, \infty)$ defined as:
\[
\zetam{\e} =  \inf\braces{\norm{x + \e y} \st \norm{x} = \norm{y} = 1, y \urcorner x}.
\]

Using Euclidean analogy, $\zetam{\e}$ is the lower bound on the ``hypotenuse'' in a ``quasi-right-angled'' triangle, where the ``cathetus'' of length $\e$ is quasi-orthogonal to another one of length 1.

We say that two functions $f$ and $g$ are \emph{equivalent} if there exist positive constants $a, b, c, d, e$ such that $a f(bt) \leq g(t) \leq c f(dt)$ for all $t \in [0, e]$. It follows from \cite[Corollary 2]{ivanov2017new} that the function $\zetam{\cdot} - 1$ is equivalent to the modulus of convexity:

\begin{prp}
\label{prp:hypotenuse_mcox}
Let $X$ be an arbitrary Banach space. Then the functions $\zetam{\cdot} - 1$ and $\mcox{\cdot}$ are equivalent, and the following inequalities hold:
\[
\mcox{\frac{\e}{2}}  \leq \zetam{\e} - 1 \leq \mcox{2\e}, \quad \e \in \left[0, 1\right].
\]
\end{prp} 

\begin{rem}
There is a typo in \cite[Corollary 2]{ivanov2017new}, where it is written as $\zetam{\e}$ instead of $\zetam{\e} - 1$ in the corresponding inequality. Furthermore, the bound given there is $ \mcox{\frac{\e}{1 + \e}} \leq \zetam{\e} - 1$, which implies the bound $\mcox{\frac{\e}{2}} \leq \zetam{\e}$ on the given interval $\e \in \left[0, 1\right]$.
\end{rem}

While we would love to claim that the function $\zetam{\cdot}$ is convex, we do not believe this is true in general. Since the modulus of convexity is not necessarily convex (see \emph{Remarks} on p.67 of \cite{Lindenstrauss1979} and the original paper \cite{Liokoumovich1973}). We will return to the properties of the function $\zetam{\cdot}$ in \Href{Section}{sec:hypotenuse}, as we will need several technical observations. These observations are based on simple two-dimensional geometry.

We will also use the following property of uniformly convex Banach spaces \cite[Theorem 8.2.2.]{larsen1973functional}.

We say that a point $p \in K$ is a \emph{nearest point} from a point $x$ to $K$ in a given Banach space $\parenth{X, \norm{\cdot}}$ if $\dist{x}{K} = \norm{x - p}$.

\begin{prp}
\label{prp:nearest_point_ucs}
Let $K$ be a non-empty closed convex subset of a uniformly convex Banach space. For any point $x$, there exists a unique nearest point from $x$ to $K$.
\end{prp}

\section{Helly-type results and their proofs}
\label{sec:proofs_no-dim_Helly}
In this section, we will prove several no-dimensional Helly-type results in uniformly convex Banach spaces. We will see that the combinatorial part of the proofs of the following Helly-type results is exactly the same as in the Euclidean case in \cite{adiprasito2020theorems}, and the only difference will be an extension of the Euclidean inequality of \Href{Proposition}{prp:euclidean_min_deviation}.

\subsection{Assumptions and Notations}
\label{subsec:assumptions}
For a fixed normed space $X,$
 the sequence $\{r_j\}$ is defined inductively:
\begin{itemize}
\item $r_1 =1;$
\item for all natural $j,$ $r_{j+1}$ is defined as the unique solution to 
\begin{equation}
\label{eq:bump_sequence_Helly}
 r_{j+1} \cdot \zetam{r_{j+1}}= {r_j}.
\end{equation}
\end{itemize}

The sequence $\{r_k\}$ will play the role of $\braces{r_k(X)}$ from \Href{Theorem}{thm:no-dim_Helly_Banach}, \Href{Theorem}{thm:no-dim_fractional_colorful_Helly}, and \Href{Theorem}{thm:no-dim_centerpoint}. We will show that it is well-defined later in \Href{Section}{sec:hypotenuse}.

Throughout this section, we assume that $X$ is a uniformly convex Banach space.

\subsection{Arithmetic and geometric properties of the sequence $r_k$}

We now formulate two statements about the properties of the sequence $\{r_k\}$ defined in \Href{Subsection}{subsec:assumptions}, postponing their proofs.

We will use the following arithmetic properties of the sequence $\{r_k\}$. The proof is based on routine analysis.

\begin{lem}\label{lem:bump_sequence}
In any normed space $X$, the sequence $\{r_k\}$ is well-defined. 
If $X$ is uniformly convex, then $\{r_k\}$ is a strictly monotonically decreasing null sequence. 
Additionally, if the modulus of convexity of $X$ satisfies inequality \eqref{eq:mcox_power_type}, that is, it is of power type $q$, then
\[
r_k \leq \tilde{C}_r k^{-\frac{1}{q}},
\]
where $\tilde{C}_r = \max\!\braces{1, 2\parenth{\frac{2}{q C_X}}^{\frac{1}{q}}}$.
\end{lem}

This purely technical lemma will be proven in \Href{Section}{sec:hypotenuse}. 
\begin{rem}
In the Euclidean case, \Href{Lemma}{lem:bump_sequence} gives the bound
 $\tilde{C}_r = 2 \parenth{\sqrt{2}+1}^{\frac{1}{2}} < 4$. This means that our method gives us the bound $r_k(\R^d) = \frac{4}{\sqrt{k}}$ in \Href{Proposition}{prp:no-dim_Helly}.
\end{rem}

Our non-Euclidean analogue of \Href{Proposition}{prp:euclidean_min_deviation} is the following, which is a more or less direct corollary of \Href{Lemma}{lem:deviation_from_supp_hyperplane}:

\begin{cor}
\label{cor:min_deviation_from_hyperplane_sequence}
Fix $j \geq 2.$ Under the assumptions in \Href{Subsection}{subsec:assumptions}, assume that a point $p$ and two convex sets $K$ and $Q$ in $X$ satisfy the following conditions:
\begin{enumerate}
\item $p$ is the nearest point from the origin to $K$, and
$\dist{0}{K} = \norm{p} > \frac{1}{r_{j-1}}$; 
\item $\dist{p}{L}  > 1$.
\end{enumerate}
Then $\dist{0}{K \cap L} > \frac{1}{r_j}$, with the convention that $\dist{q}{\emptyset} = \infty$.
\end{cor}

This corollary will be proven after the proof of \Href{Lemma}{lem:deviation_from_supp_hyperplane} in \Href{Section}{sec:deviation_from_hyperplane}. Additionally, we will formulate a more general version of \Href{Corollary}{cor:min_deviation_from_hyperplane_sequence} without specific bounds on $\norm{p}$ (i.e., without involving $r_{j-1}$ and $r_j$).

\begin{rem}
The author wishes to emphasize again that there is no essential difference between the proofs of the corresponding Helly-type results in \cite{adiprasito2020theorems} and \cite{adiprasito2019theorems} for the Euclidean case and the uniformly convex Banach space case considered here, except for the use of \Href{Corollary}{cor:min_deviation_from_hyperplane_sequence}.
\end{rem}

\subsection{No-dimensional fractional Helly theorem and its corollaries}
The fractional version of the classical Helly theorem was obtained in \cite{katchalski1979problem}. Following the approach outlined in \cite{adiprasito2020theorems} for the Euclidean case, the no-dimensional Helly's theorem follows from the no-dimensional fractional version stated below:

\begin{thm}[No-dimensional fractional Helly's theorem]\label{thm:no-dim_fractional_Helly_Banach}
Let $\alpha \in (0,1]$ and $\beta = 1-(1-\alpha)^{\frac{1}{k}}$. Under the assumptions in \Href{Subsection}{subsec:assumptions}, if $\mathcal{F}$ is a finite family of at least $k$ convex sets in $X$ such that for an $\alpha$ fraction of $k$-tuples $K_1,\ldots,K_k$ of $\mathcal{F}$, the set $\bigcap\limits_{i \in [k]} K_i$ contains a point in the unit ball $\ball{}$, then there exists $x \in X$ such that at least $\beta \card{\mathcal{F}}$ elements of $\mathcal{F}$ intersect the ball $r_k \ball{} + x$.
\end{thm}

\begin{proof}
Set $n = \card{\mathcal{F}}$. We scale the picture by a factor of $\frac{1}{r_k}$. That is, we assume that $\mathcal{F}$ is a finite family of at least $k$ convex sets in $X$ such that for an $\alpha$ fraction of $k$-tuples $K_1,\ldots,K_k$ of $\mathcal{F}$, the set $\bigcap\limits_{i \in [k]} K_i$ has a point in the ball $\frac{1}{r_k}\ball{}$, and we want to show that there is $x \in X$ such that at least $\beta n$ elements of $\mathcal{F}$ intersect the ball $\ball{} + x$.

For contradiction, assume that for any $p \in X$, there are at least $m$ sets $K$ in $\mathcal{F}$, for $m > (1 - \beta)n$, such that $\dist{p}{K} > 1$. Without loss of generality, we assume that all the sets are closed. Using \Href{Proposition}{prp:nearest_point_ucs}, under this assumption, one sees that for any point $x$ and any intersection $S$ of a finite number of sets in $\mathcal{F}$, there is a nearest point from $x$ to $S$.

Call the ordered $j$-tuple ${K_1, \dots, K_j}$ \emph{good} if $\dist{0}{\bigcap\limits_{i \in [j]}{K_i}} > \frac{1}{r_j}. $ We show, by induction on $j$, that the number of good $j$-tuples ${K_1, \dots, K_j}$ of $\mathcal{F}$ is greater than $(1-\beta)^j n (n-1) \dots (n-j+1)$. Note that $n(n - 1) \dots (n - j + 1)$ is the total number of ordered $j$-tuples of $\mathcal{F}$.

Setting $j = k$, we get that the fraction of ordered $k$-tuples that are good, i.e., $k$-tuples whose common intersection is disjoint from $\frac{1}{r_k}\ball{}$, is greater than $(1 - \beta)^k = (1 - \alpha)$, a contradiction to our assumption on the number of $k$-tuples intersecting $\frac{1}{r_k}\ball{}$.

Fix a good $(j-1)$-tuple ${K_1, \dots, K_{j-1}}$ of $\mathcal{F}$. Consider two cases:

\medskip
\noindent \textbf{Case 1.} The intersection $\bigcap\limits_{i \in [j-1]}{K_i}$ is empty. Then we can add any $K \in \mathcal{F}$ distinct from $K_1, \dots, K_{j-1}$. This gives $n - j + 1$ good $j$-tuples of $\mathcal{F}$ extending each previous good $(j - 1)$-tuple.

\medskip
\noindent \textbf{Case 2.} The intersection $\bigcap\limits_{i \in [j-1]}{K_i}$ is non-empty. Let $p$ be the nearest point in $\bigcap\limits_{i \in [j-1]}{K_i}$ to the origin. As we are considering good $(j -1)$-tuples, we have $\norm{p} > \frac{1}{r_{j-1}}$. Next, there are $m$ sets $K_j \in \mathcal{F}$ with $\dist{p}{K_j} > 1$. All these sets are different from all $K_i, i \in [j-1]$, because $p$ belongs to all $K_i, i \in [j-1]$. Fix one such $K_j.$  

Applying \Href{Corollary}{cor:min_deviation_from_hyperplane_sequence}, we have:
\[
\dist{0}{\bigcap\limits_{i \in [j]}{K_i}}= 
\dist{0}{\parenth{\bigcap\limits_{i \in [j-1]}{K_i}} \cap K_{j}} > \frac{1}{r_j}.
\]
Thus, $K_1, \dots, K_j$ is a good $j$-tuple.
It gives altogether $m$ $j$-tuples of $\mathcal{F}$ extending each previous good $(j-1)$-tuple.

Hence, each good $(j-1)$-tuple can be extended to a good $j$-tuple in either $m$ or $n -j+1$ ways, completing the induction.
\end{proof}

\begin{thm}\label{thm:no-dim_Helly_Banach_uco}
Under the assumptions in \Href{Subsection}{subsec:assumptions}, if $K_1,\ldots,K_n$ are convex sets in $X$ such that the unit ball $\ball{}$ intersects $\bigcap\limits_{j \in J}K_j$ for every $J\subset [n]$ of size $k$, then there is a point $x \in X$ such that
\[
\dist{x}{K_i} \leq r_k \quad \text{for all} \quad i \in [n].
\]
\end{thm}

\begin{proof}
The result directly follows from \Href{Theorem}{thm:no-dim_fractional_Helly_Banach} with $\alpha = 1$.
\end{proof}

\subsection{No-dimensional colorful fractional Helly theorem and its corollaries}
The now-classical extension of the Helly theorem is its colorful version of Lov\'asz \cite{barany1982generalization}. The colorful version of the fractional Helly theorem, originally due to Katchalski and Liu, was obtained in \cite{barany2014colourful}. Here, we will follow the approach suggested in \cite{adiprasito2020theorems} to derive a no-dimensional colorful Helly theorem from its fractional version.
\begin{thm}\label{thm:genhelly}
Under the assumptions in \Href{Subsection}{subsec:assumptions}, let $\mathcal{F}_1, \dots, \mathcal{F}_k$ be finite and non-empty families of convex sets in $X$. If for every $p \in X$ there are at least $m_i$ sets $K \in \mathcal{F}_i$ such that $\dist{p}{K} > 1$ for all $i \in [k]$, then for every $q \in X$, there are at least $\prod\limits_{i \in [k]} m_i$ ``rainbow'' $k$-tuples $K_1 \in \mathcal{F}_1, \dots, K_k \in \mathcal{F}_k$ such that
\[
\dist{q}{\bigcap\limits_{j \in [k]} K_j} > \frac{1}{r_k},
\]
with the convention that $\dist{q}{\emptyset}=\infty$.
\end{thm}

\begin{proof}
Without loss of generality, assume that all the sets are closed. Using \Href{Proposition}{prp:nearest_point_ucs}, one sees that for any point $x$ and any intersection $S$ of a finite number of sets in $\mathcal{F}_1, \dots, \mathcal{F}_k$, there is a nearest point from $x$ to $S$.

The proof proceeds by induction on $k$, where the base case $k = 1$ is trivial. For the induction step from $k-1$ to $k$, fix a point $x \in X$ and consider the families $\mathcal{F}_1, \dots, \mathcal{F}_{k-1}$. By the induction hypothesis, there are at least $\prod\limits_{i \in [k-1]} m_i$ ``rainbow'' $(k-1)$-tuples $K_1 \in \mathcal{F}_1, \dots, K_{k-1} \in \mathcal{F}_{k-1}$ such that:
\[
\dist{x}{\bigcap\limits_{j \in [k-1]} K_j} > \frac{1}{r_{k-1}}.
\]
Similar to the proof of \Href{Theorem}{thm:no-dim_fractional_Helly_Banach}, consider the following two cases:

\medskip
\noindent \textbf{Case 1.} The intersection $\bigcap\limits_{i \in [k-1]} K_i$ is empty. In this case, adding any $K_k \in \mathcal{F}_k$ results in $\bigcap\limits_{i \in [k]} K_i = \emptyset$. This implies that the $(k-1)$-tuple $K_1 \in \mathcal{F}_1, \dots, K_{k-1} \in \mathcal{F}_{k-1}$ can be extended to a ``rainbow'' $k$-tuple that satisfies the desired inequality in $\card{\mathcal{F}_k}$ different ways.

\medskip
\noindent \textbf{Case 2.} The intersection $\bigcap\limits_{i \in [k-1]} K_i$ is non-empty. Let $p$ be the nearest point in $\bigcap\limits_{i \in [k-1]} K_i$ to $x$. By assumption, there are at least $m_k$ sets $K_k \in \mathcal{F}_k$ such that $\dist{p}{K_k} > 1$. Applying \Href{Corollary}{cor:min_deviation_from_hyperplane_sequence}, we have:
\[
\dist{0}{\bigcap\limits_{i \in [k]}{K_i}}= 
\dist{0}{\parenth{\bigcap\limits_{i \in [k-1]}{K_i}} \cap K_{k}} > \frac{1}{r_k}.
\]

In both cases, the $(k-1)$-tuple $K_1 \in \mathcal{F}_1, \dots, K_{k-1} \in \mathcal{F}_{k-1}$ can be extended to a ``rainbow'' $k$-tuple with $K_k \in \mathcal{F}_k$ in at least $m_k$ ways, meaning that there are at least $\prod\limits_{i \in [k]} m_i$ such $k$-tuples satisfying the inequality
\[
 \dist{x}{\bigcap\limits_{j \in [k]} K_j} > \frac{1}{r_k}.
\]
\end{proof}

\begin{thm}[No-dimensional colorful fractional Helly's theorem]
\label{thm:no-dim_fractional_colorful_Helly_Banach_uco}
Let $\alpha \in (0,1]$ and $\beta = 1-(1-\alpha)^{\frac{1}{k}}$. Under the assumptions in \Href{Subsection}{subsec:assumptions}, let $\mathcal{F}_1, \dots, \mathcal{F}_k$ be finite and non-empty families of convex sets in $X$. If for an $\alpha$ fraction of ``rainbow'' $k$-tuples $K_1 \in \mathcal{F}_1, \dots, K_k \in \mathcal{F}_k$, the set $\bigcap\limits_{i \in [k]} K_i$ contains a point in the unit ball $\ball{}$, then there exist $x \in X$ and $i \in [k]$ such that at least $\beta |\mathcal{F}_i|$ elements of $\mathcal{F}_i$ intersect the ball $r_k \ball{} + x$.
\end{thm}

\begin{proof}
The theorem is a direct consequence of \Href{Theorem}{thm:genhelly}. Denote the size of $\mathcal{F}_i$ by $n_i$. If no ball $r_k \ball{} + x$ intersects $\beta n_i$ elements of $\mathcal{F}_i$, 
then for every $p \in X$, there are strictly more than $(1 - \beta) n_i$ sets $K \in \mathcal{F}_i$ such that $\dist{p}{K} > r_k$. Denote $m_i = (1 - \beta) n_i$. If this assumption holds for all $i \in [k]$, then the number of ``rainbow'' $k$-tuples that are disjoint from a fixed unit ball is larger than:
\[
\prod\limits_{i \in [k]} m_i > (1 - \beta)^k \prod\limits_{i \in [k]} n_i = (1 - \alpha)\prod\limits_{i \in [k]} n_i,
\]
which contradicts the assumption that at least an $\alpha$ fraction of ``rainbow'' $k$-tuples intersect the ball $\ball{}$.
\end{proof}

\begin{thm}
\label{thm:no-dim_colorful_Helly_Banach_uco}
Under the assumptions in \Href{Subsection}{subsec:assumptions}, let $\mathcal{F}_1, \dots, \mathcal{F}_k$ be finite and non-empty families of convex sets in $X$. If for every ``rainbow'' $k$-tuple $K_1 \in \mathcal{F}_1, \dots, K_k \in \mathcal{F}_k$, the set $\bigcap\limits_{i \in [k]} K_i$ contains a point in the unit ball $\ball{}$, then there exist $x \in X$ and $i \in [k]$ such that all the sets in $\mathcal{F}_i$ intersect the ball $r_k \ball{} + x$.
\end{thm}

\begin{proof}
This result directly follows from 
\Href{Theorem}{thm:no-dim_fractional_colorful_Helly_Banach_uco} with $\alpha = 1$.
\end{proof}

\subsection{Proofs of the results from the Introduction}
Using \Href{Proposition}{prp:superreflexivity_uniform_conv}, we consider an equivalent uniformly convex norm on $X$ whose modulus of convexity is of some power type $q$. \Href{Theorem}{thm:no-dim_Helly_Banach} and \Href{Theorem}{thm:no-dim_fractional_colorful_Helly} then follow from their versions in uniformly convex Banach spaces and the bound on $r_k$ obtained in \Href{Lemma}{lem:bump_sequence}.

\begin{proof}[Proof of \Href{Theorem}{thm:no-dim_Helly_Banach}]
The result directly follows from \Href{Theorem}{thm:no-dim_Helly_Banach_uco}.
\end{proof}

\begin{proof}[Proof of \Href{Theorem}{thm:no-dim_fractional_colorful_Helly}]
The result directly follows from \Href{Theorem}{thm:no-dim_fractional_colorful_Helly_Banach_uco}.
\end{proof}

The proof of the no-dimensional centerpoint theorem is straightforward. We use the argument presented in the proof of Theorem 2.10 in \cite{adiprasito2019theorems}, adjusted for our situation.

\begin{proof}[Proof of \Href{Theorem}{thm:no-dim_centerpoint}]
Let $\mathcal{F}$ be the family of convex hulls of more than $\frac{k-1}{k} n$ points of $P$. By a standard double-counting argument, any $k$-tuple of sets from $\mathcal{F}$ contains a point of $P$ in common. Furthermore, since each point of $P$ lies in $\ball{}$, the common intersection of every $k$-tuple of $\mathcal{F}$ must intersect $\ball{}$. Applying \Href{Theorem}{thm:no-dim_Helly_Banach}, we obtain the existence of a point $x$ and a sequence $\{r_k(X)\}$ with the desired property such that $r_k(X) \ball{} + x$ intersects every element of $\mathcal{F}$. This is the required no-dimensional centerpoint.

Indeed, consider any closed half-space $h^{+}$ containing $r_k(X) \ball{} + x$. If it contains fewer than $\frac{n}{k}$ points of $P$, then its complement, the open half-space $h^{-}$, contains more than $n - \frac{n}{k} = \frac{k-1}{k} n$ points, say the set $P^\prime$, of $P$. Since $\dist{x}{P^\prime} > r_k(X)$, this contradicts the choice of $x$.
\end{proof}

\section{Properties of the Hypotenuse}
\label{sec:hypotenuse}
In the current section, we will obtain basic simple properties of the function $\zetam{\cdot}$ and prove  \Href{Lemma}{lem:bump_sequence}.

Let us introduce yet another notation, which drastically simplifies the proofs. 

Fix $\ell$ as a supporting line at $u$ to the unit ball of a normed space $X$ with a unit directional vector $v$. 

Denote $\zeta(t) = \norm{u + tv}$ for $t \in \R.$
Clearly, 
$\zetam{t} = \inf \zeta(t),$
where the infimum is taken over all  $u$ and quasi-orthogonal to $u$ unit vector $v.$

The following lemma directly follows from the fact that $\zeta(t)$ is the restriction of the norm on the line $\ell.$
\begin{lem}
\label{lem:monotonicity_zeta}
Under the above notation, the function $\zeta(\cdot)$ is a convex function satisfying the inequalities 
\[
\frac{\zeta(t_1) -1}{t_1} \leq \frac{\zeta(t_2) -1}{t_2} 
\quad \text{and} \quad \zeta(t_2) \leq \zeta(t_1) + t_2 -t_1.
\] 
for every $0 < t_1 < t_2.$   
\end{lem}

\begin{cor}
\label{cor:monotonicity_zetam}
In any normed space $X$, the function $\zetam{\cdot}$ is a continuous function satisfying the inequality
\[
\frac{\zetam{t_1} -1}{t_1} \leq \frac{\zetam{t_2} -1}{t_2} 
\] 
for every $0 < t_1 < t_2$.  

In particular, the function $\zetam{\cdot}$ is strictly monotonically increasing on 
\begin{itemize}
\item its Lebesgue level-set  $\braces{t \st \zetam{t} > 1};$
\item  the whole half-line $[0, \infty)$ if $X$ is uniformly convex.
\end{itemize} 
\end{cor}
\begin{proof}
Passing to the infimums in the leftmost inequality of \Href{Lemma}{lem:monotonicity_zeta}, we get the desired inequality. 

 The monotonicity of $\zetam{\cdot}$ on $[0, \infty)$ follows, as well as the strict monotonicity on the Lebesgue level-set  $\braces{t \st \zetam{t} > 1}.$ 
By \Href{Proposition}{prp:hypotenuse_mcox}, the Lebesgue level-set  $\braces{t \st \zetam{t} > 1}$ coincides with $(0, \infty)$ in the case when $X$ is uniformly convex. 

Passing to the infimums in the leftmost inequality of \Href{Lemma}{lem:monotonicity_zeta} and by monotonicity, 
 \[
 \zetam{t_1} \leq \zetam{t_2} \leq \zetam{t_1} + t_2 - t_1.
 \]
The continuity of $\zetam{\cdot}$ follows.
\end{proof}

\begin{proof}[Proof of \Href{Lemma}{lem:bump_sequence}]
By \Href{Corollary}{cor:monotonicity_zetam}, the function $t \zetam{t}$ is strictly monotone and continuous on $[0,1]$ in any normed space. Also, $ t \zetam{t}\mid_{t=0} = 0$ and
$ t \zetam{t}\mid_{t=1} \geq t \mid_{t=1} = 1.$ 
It follows that there is always a unique solution to the equation
$t \zetam{t} = r$ for any $r \in (0,1].$ That is, the sequence $\{r_k\}$ is well-defined. 

Assume the space $X$ is uniformly convex.  Then, the strict monotonicity of $\zetam{\cdot},$ obtained in  \Href{Corollary}{cor:monotonicity_zetam}, implies that $\{r_k\}$ is a strictly monotonically decreasing null sequence.

Now, assume inequality \eqref{eq:mcox_power_type} holds. Then, by \Href{Proposition}{prp:hypotenuse_mcox} and by the monotonicity of $\{r_k\}$, we have:
\begin{equation}
\label{eq:zeta_polynomial}
\zetam{r_{j+1}} \geq 1 + \frac{C_X}{2^q} r_{j+1}^{q} \quad \text{for all} \quad j \in \N. 
\end{equation}

We use induction on $j$ to prove the inequality $r_j \leq \tilde{C}_r j^{-\frac{1}{q}}$. Clearly, $r_1  \leq \tilde{C}_r$. The step from $j$ to $j + 1$ follows from routine analysis. 
By the monotonicity, it suffices to show 
\[
\tilde{C}_r  {(j+1)}^{-\frac{1}{q}} \zetam{\tilde{C}_r  {(j+1)}^{-\frac{1}{q}}} \geq 
\tilde{C}_r  {j}^{-\frac{1}{q}},
 \]
 which is equivalent to
 \begin{equation}
\label{eq:r_k_comp}
 \zetam{\tilde{C}_r  {(j+1)}^{-\frac{1}{q}}} \geq 
 \parenth{1 +\frac{1}{j}}^{\frac{1}{q}}
\end{equation}
Using \eqref{eq:zeta_polynomial} for the left-hand side, we get
\[
 \zetam{\tilde{C}_r  {(j+1)}^{-\frac{1}{q}}}
\geq 1 + \frac{C_X}{2^q} \frac{\tilde{C}_r^q}{j+1}.
\]

Using Bernoulli's inequality $(1 + x)^r \leq 1 + rx$ for $x> 0,$ $r \in (0,1]$ for the right-hand side of \eqref{eq:r_k_comp}, we get  
$1 + \frac{1}{qj} \geq  \parenth{1 +\frac{1}{j}}^{\frac{1}{q}}.$

Thus, inequality \eqref{eq:r_k_comp} follows from the inequality
\[
 1 + \frac{C_X}{2^q} \frac{\tilde{C}_r^q}{j+1} 
 \geq 
1 + \frac{1}{qj} 
\quad \Leftrightarrow \quad
 \tilde{C}_r \geq  
 2 \parenth{\frac{j+1}{j} \cdot \frac{1}{q C_X}}^{\frac{1}{q}} 
\]
which follows from the inequality $\tilde{C}_r \geq 2\parenth{\frac{2}{q C_X}}^{\frac{1}{q}}$ for natural $j$. The proof of \Href{Lemma}{lem:bump_sequence} is complete.
\end{proof}

\section{Which side are you on?}
\label{sec:deviation_from_hyperplane}

In this section, we study the deviations of supporting hyperplanes to the unit ball from the unit ball itself. We start by proving \Href{Lemma}{lem:deviation_from_supp_hyperplane}. Next, we will study the minimal deviation in the outward-facing half-space and obtain several technical facts \Href{Corollary}{cor:min_deviation_from_hyperplane_sequence}.  Then, we explore the interrelations between the identity for supremums of \Href{Lemma}{lem:deviation_from_supp_hyperplane} and the no-dimensional Carath\'eodory theorem.

\subsection{Proof of  \Href{Lemma}{lem:deviation_from_supp_hyperplane}}
We will need the following reduction lemma. We will say that a Banach space is \emph{strictly convex} and \emph{smooth} if the unit sphere does not contain a segment and there is exactly one supporting hyperplane at every boundary point of the unit sphere.

\begin{lem}
\label{lem:reduction_supp_h_deviation}
It suffices to prove \Href{Lemma}{lem:deviation_from_supp_hyperplane} in the case when $X$ is a two-dimensional strictly convex and smooth space.
\end{lem}

\begin{proof}
The identities in \Href{Lemma}{lem:deviation_from_supp_hyperplane} follow if they hold for every two-dimensional plane passing through $x$. Thus, it suffices to consider the case when $X$ is two-dimensional.

In the two-dimensional case, the infimum and the supremum can be replaced by the minimum and the maximum by standard compactness and inclusion arguments:
\begin{equation}
\label{eq:infsup_to_minmax}
\inf\limits_{
y \in H^{+}  \st \norm{x-y} \geq \rho} \norm{y}
=
\min\limits_{
y \in H^{+}  \st \norm{x-y} = \rho} \norm{y} 
\quad \text{and} \quad 
\sup\limits_{
z \in H^{-}  \st \norm{x-z} \leq \rho} \norm{z}
=
\max\limits_{
z \in H^{-}  \st \norm{x-z} = \rho} \norm{z}.
\end{equation}

The unit ball of the norm $\norm{\cdot}$ can be approximated by the unit ball in a strictly convex and smooth space with any desired precision (for example, see \cite[Theorem 2.7.1]{schneider2014convex}).
 Consider a strictly convex and smooth space $\parenth{X, \norm{\cdot}_\tau}$ whose norm satisfies 
\begin{equation}
\label{eq:mcox_hyperplane_bound_equiv_norms}
\norm{a}_\tau \leq \norm{a} \leq (1+C \tau )\norm{a}_\tau
\end{equation}
holds for all $a \in X, \tau > 0$ and some positive constant $C$. 

The leftmost inequality in \eqref{eq:mcox_hyperplane_bound_equiv_norms} implies that the unit ball of $\parenth{X, \norm{\cdot}_\tau}$ contains the unit ball of $X.$
We use $H_{\tau}$ to denote the line parallel to $H$ contained in $H^{+}$ and supporting the unit ball of the norm $\norm{\cdot}_\tau$.   Let $x_\tau$ be the point at which $H_\tau$ supports the unit ball of the norm $\norm{\cdot}_\tau$. Let $v_\tau$ and $w_\tau$ be points on $H_\tau$ with the property $\norm{x_\tau - v_\tau}_\tau = \norm{x_\tau - w_\tau}_\tau =  \rho$.

Suppose the minimum in \eqref{eq:infsup_to_minmax} is not attained on the supporting line $H$, i.e., there is a $\tilde{y}$ in the interior of $H^+$ such that $\norm{\tilde{y} - x} = \rho$ and 
\begin{equation}
\label{eq:min_mcox_hyperplane_bound}
\norm{\tilde{y}}  < \min\{\norm{v}, \norm{w}\},
\end{equation}
where $v$ and $u$ belong to $H$ and satisfy $\norm{x - v} = \norm{x - w} = \rho$. Let $y_\tau$ be a vector of the form $ \lambda \tilde{y}$, with $\lambda > 0$, such that $\norm{x_\tau - y_\tau}_\tau = \rho$. Using the equivalence of the norms \eqref{eq:mcox_hyperplane_bound_equiv_norms}, $x_{\tau}, w_\tau, v_\tau,$ and $y_\tau$ change continuously in $\tau$. Hence, for sufficiently small $\tau$, the inequality $\norm{y_\tau}_\tau < \min\{\norm{v_\tau}_\tau, \norm{w_\tau}_\tau\}$ holds. Thus, if \eqref{eq:min_mcox_hyperplane_bound} does not hold, it does not hold in some two-dimensional smooth space.

Similarly, if there is a point $\tilde{z}$ in the interior of $H^-$ such that $\norm{\tilde{z} - x} = \rho$ and 
\[
\norm{\tilde{z}}  > \max\{\norm{v}, \norm{w}\},
\]
then such an example exists in some two-dimensional strictly convex and smooth space. 
The lemma follows.
\end{proof}

\begin{proof}[Proof of \Href{Lemma}{lem:deviation_from_supp_hyperplane}]
Using \Href{Lemma}{lem:reduction_supp_h_deviation}, we assume that $X$ is two-dimensional, strictly convex, and smooth. We will use a simple geometric observation: if two balls are supported by a hyperplane at a point $p$ in a strictly convex and smooth Banach space, then the centers of the balls and the point $p$ must lie on the same line.

First, assume the infimum is attained at some $\tilde{y}$ in the interior of $H^{+}$. In a smooth space, this implies that the balls $\norm{\tilde{y}} \ball{}$ and $\rho \ball{} + x$ have the same supporting hyperplane at $\tilde{y}$. Since $\tilde{y} \in H^{+}$, $\norm{\tilde{y}} \geq \norm{x} = 1$, and consequently, $x \in \norm{\tilde{y}} \ball{}$. It follows that $x$ belongs to the line passing through the origin and $\tilde{y}$. Thus, $\norm{\tilde{y}} = 1 + \rho$, which is the maximum norm of a vector from the set $\braces{y \in H^{+}  \st \norm{x-y} = \rho}$ by the triangle inequality. The identity for the infimum in \Href{Lemma}{lem:deviation_from_supp_hyperplane} follows.

Now, assume the supremum is attained at some $\tilde{z}$ in the interior of $H^{-}$. In a smooth space, this implies that the balls $\norm{\tilde{z}} \ball{}$ and $\rho \ball{} + x$ have the same supporting hyperplane at $\tilde{z}$. If $\rho \leq 1$, then the origin is not in the ball $\rho \ball{} + x$. By convexity, this means that $\norm{\tilde{z}} = 1 - \rho$. If $\rho > 1$, then the origin is in the ball $\rho \ball{} + x$. By convexity, this means that $\norm{\tilde{z}} = \rho - 1$. In both cases, $\enorm{\rho - 1}$ is the minimal norm of a vector from the set $\braces{z \in H^{-}  \st \norm{x-z} = \rho}$. The identity for the supremum follows.

This completes the proof of \Href{Lemma}{lem:deviation_from_supp_hyperplane}.
\end{proof}

\subsection{The minimal deviation in the outward-facing half-space}
\label{subsec:min_dev}

In the current subsection, we will derive  \Href{Corollary}{cor:min_deviation_from_hyperplane_sequence} from \Href{Lemma}{lem:deviation_from_supp_hyperplane}.

We start by proving using \Href{Lemma}{lem:deviation_from_supp_hyperplane} the following more general version of \Href{Corollary}{cor:min_deviation_from_hyperplane_sequence}:

\begin{lem}
\label{lem:min_deviation_from_hyperplane}
Assume that strictly positive numbers $\rho$ and  $\rho_L,$ two convex sets $K$ and $L$, and points $x$ and  $y$ in a normed space satisfy the following conditions:
\begin{enumerate}
\item $y$ is a nearest point from $x$ to $K$ and
$\dist{x}{K} = \norm{x - y} = \rho;$ 
\item $\dist{y}{K \cap L}  > \rho_L;$
\end{enumerate}
Then $\dist{x}{K \cap L} \geq  \rho \cdot \zetam{\frac{\rho_L}{\rho}}$. Moreover, if  $X$ is uniformly convex, the inequality is strict. 
\end{lem}
\begin{proof}
It is nothing to prove if the intersection ${K \cap L}$ is empty. Assume ${K \cap L}$ is non-empty.
Since $\dist{x}{K \cap L} \geq \dist{x}{K},$  the desired non-strict inequality follows in the case 
$\zetam{\frac{\rho_L}{\rho}} = 1.$ 

So, consider the case $\zetam{\frac{\rho_L}{\rho}} >1.$
Shift the picture so that $x$ is at the origin, and scale it by $\frac{1}{\rho}$. 
Consider arbitrary $z \in K \cap L.$ 
Consider the two-dimensional plane $H_2$ spanned by $y$ and $z$ (if $y$ and $z$ belong to a single line, choose as $H_2$  any two-dimensional plane passing through this line ).

By a simple separation argument and convexity, there exists a line $\ell$ supporting the unit ball of $\parenth{H_2, \norm{\cdot}}$ at $y$ and separating $x$ and $z$ in $H_2$.  Hence, \Href{Lemma}{lem:deviation_from_supp_hyperplane} and the definition of $\zetam{\cdot}$ yield:
\[
\norm{z} \geq  \rho \cdot \zetam{\frac{\norm{z - y}}{\rho}}.
\]
Using \Href{Corollary}{cor:monotonicity_zetam} twice, we get 
 $\dist{x}{K \cap L} \geq  \rho \cdot \zetam{\frac{\dist{y}{K \cap L}}{\rho}} > \rho \cdot \zetam{\frac{\rho_L}{\rho}} $.
 
 In the uniformly convex case, $\zetam{\frac{\rho_L}{\rho}} > 1$ by  \Href{Corollary}{cor:monotonicity_zetam}.
Thus, the desired strict inequality follows. 
\end{proof}

Note that in the Euclidean case, $\zetam{\e} = \sqrt{1 + \e^2}$ and the bound in  \Href{Lemma}{lem:min_deviation_from_hyperplane} is $\sqrt{\rho_K^2 + \rho_L^2}$. Clearly, the bound is optimal.

\begin{proof}[Proof of \Href{Corollary}{cor:min_deviation_from_hyperplane_sequence}]
Applying \Href{Lemma}{lem:min_deviation_from_hyperplane}, we have:
\begin{equation}
\label{eq:main_banach_space_ineq}
\dist{0}{K \cap L} > \norm{p} \zetam{\frac{1}{\norm{p}}}.
\end{equation}
If $\norm{p} \geq \frac{1}{r_{j}},$ then $\dist{0}{K \cap L} > \frac{1}{r_j}.$
Consider the case $\norm{p} < \frac{1}{r_{j}}$ or, equivalently,
$\frac{1}{\norm{p}} > {r_{j}}.$
Then, from \eqref{eq:main_banach_space_ineq}, one gets:
\[
\dist{0}{K \cap L} > \norm{p} \zetam{\frac{1}{\norm{p}}} =
\norm{p} \parenth{\zetam{\frac{1}{\norm{p}}} -1} + \norm{p} 
{\geq} 
\frac{\zetam{r_j} - 1}{r_j} +\norm{p} ,
\]
where the last inequality follows from \Href{Corollary}{cor:monotonicity_zetam} and the assumption  $\frac{1}{\norm{p}} > {r_{j}}.$ 
By \Href{Lemma}{lem:bump_sequence},  $r_{j-1} > r_j$ in a uniformly convex space. 
Hence, $\frac{1}{r_j} > \frac{1}{r_{j-1}};$ also by assertion of the corollary
$\norm{p} > \frac{1}{r_{j-1}}.$  Using this inequality, one gets
\[
\dist{0}{K \cap L} > \frac{\zetam{r_j} - 1}{r_j} +\norm{p} > 
\frac{\zetam{r_j} - 1}{r_{j-1}} + 
\frac{1}{r_{j-1}} = \frac{\zetam{r_j}}{r_{j-1}} = \frac{1}{r_{j}}.
\]
\end{proof}

\subsection{The maximal deviation in the inward-facing half-space}
Our goal is to show that the identity involving the supremums in \Href{Lemma}{lem:deviation_from_supp_hyperplane} is essentially the key tool in the proof of the no-dimensional Carath\'eodory theorem.

From the statement of \Href{Lemma}{lem:deviation_from_supp_hyperplane}, it is clear that the infimum can be easily bounded by our modified ``modulus of convexity'' $\zetam{\cdot}$. The supremum can be bounded by the function $\zetap{\cdot}: [0,\infty) \to [0,\infty)$, defined as follows:
\[
\zetap{\e} =  
\sup\braces{\norm{x + \e y} \st \norm{x} = \norm{y} = 1, y \urcorner x}.
\]
This function was studied in \cite{ivanov2017new}. In particular, it was shown that $\zetap{\cdot} - 1$ is equivalent to the so-called \emph{modulus of smoothness} $\mglx{\cdot}: [0,\infty) \to [0,\infty)$ of a Banach space, defined as:
\begin{equation*}
\mglx{\tau} = \sup \left\{\frac{\norm{x + y} + \norm{x - y}}{2} - 1 \st \, \norm{x} = 1, \norm{y} = \tau \right\}.
\end{equation*}
The following result was shown in \cite[Corollary 2]{ivanov2017new}:
\begin{prp}
\label{prp:hypotenuse_mglx}
Let $X$ be an arbitrary Banach space. Then the functions $\zetap{\cdot} - 1$
and $ \mglx{\cdot}$ are equivalent, and the following inequalities hold:
$$
\mglx{\frac{\e}{4}}  \leq \zetap{\e} - 1 \leq \mglx{2\e}, \quad  \e \in \left[0, \frac{1}{2}\right].
$$
\end{prp}

A Banach space $X$ is called \emph{uniformly smooth} if $\mglx{t} = o(t)$ as $t \to 0$. It is known that uniform smoothness is equivalent to the uniform differentiability of the norm.

As another direct corollary of \Href{Lemma}{lem:deviation_from_supp_hyperplane}, we obtain the following:
\begin{cor}
\label{cor:max_deviation_via_mglx}
Let a continuous linear functional $p$ attain its norm on a non-zero vector $x$. Denote by $H$ the hyperplane on which $p$ vanishes. Let $H^{-}$ denote the half-space with boundary $H$ that does not contain $x$. Then for any $z \in H^{-}$, the inequality
\[
\norm{x + z} \leq \norm{x} \zetap{\frac{\norm{z}}{\norm{x}}}
\]
holds.
\end{cor}

\begin{proof}
By \Href{Lemma}{lem:deviation_from_supp_hyperplane}, it suffices to consider the case when $z$ is quasi-orthogonal to $x$. Then, the desired inequality follows directly from the definition of $\zetap{\cdot}$.
\end{proof}

This corollary improves upon Lemma 3.1 in \cite{ivanov2021approximate}, where $\norm{x + z}$ is bounded in terms of the modulus of smoothness of the space.

Let us formulate the no-dimensional Carath\'eodory theorem in a similar way to our formulation of \Href{Theorem}{thm:no-dim_Helly_Banach}.

We will say that a \emph{sequence $\braces{R_k(X)}$ satisfies the no-dimensional Carath\'eodory theorem} in a normed space $X$ if for an arbitrary subset $S$ of the unit ball containing the origin in its convex hull, there is a sequence $\{x_i\}_1^\infty \subset S$ satisfying the inequality
\[
\norm{\frac{x_1 + \dots + x_k}{k}} \leq \frac{R_k(X)}{k}.
\]

The main result of \cite{ivanov2021approximate} was as follows:
\def\Const{4e^2}
\begin{enumerate}
\item In a uniformly smooth Banach space $X$, the sequence 
\begin{equation}
\label{eq:R_k_mglx}
R_k(X) = \frac{\Const}{k \rmglx{1/k}}
\end{equation}
 satisfies the no-dimensional Carath\'eodory theorem.
\item A description of an explicit greedy algorithm for finding the desired sequence $\{x_i\}_1^\infty$ for any set $S$.
\end{enumerate}

Let us sketch the proof of the no-dimensional Carath\'eodory theorem using \Href{Corollary}{cor:max_deviation_via_mglx} and the same greedy algorithm as in 
\cite{ivanov2021approximate}.

For a fixed normed space $X$, the sequence $\{R_j\}$ is defined inductively:
\begin{itemize}
\item $R_1 =1;$
\item for all natural $j$, $R_{j+1}$ is defined as the unique solution to 
\begin{equation*}
 \frac{R_{j+1}}{ \zetap{\frac{1}{R_{j+1} - 1}}}= {R_j}.
\end{equation*}
\end{itemize}

\begin{prp}
\label{prp:no-dim_caratheodory}
In a uniformly smooth Banach space $X$, the sequence 
$\{R_k\}$
satisfies the no-dimensional Carath\'eodory theorem.
\end{prp}

\begin{proof}[Sketch of the proof] 
For an arbitrary subset $S$ of the unit ball containing the origin in its convex hull, we will construct the sequence $\{x_i\}$ satisfying the inequality 
$\norm{x_1 + \dots + x_k} \leq R_k$. 
The proof consists of two steps:
\begin{enumerate}
\item Algorithmic construction of a sequence $\{x_i\}$.
\item Estimation of the norm $\norm{{x_1 + \dots + x_k}}$.
\end{enumerate}  

The first step explicitly uses supporting hyperplanes.

We use the following \textbf{algorithm}:
\begin{enumerate}
\item $x_1$ is an arbitrary point from $S$; 
\item For the constructed sequence $\{x_1, \dots, x_{k}\}$, $k \geq 1$, set $a_k = x_1 + \dots + x_k$. We choose $x_{k+1} \in S$ such that $\iprod{x_{k+1}}{a_k^\ast} \leq 0$, where $a_k^\ast$ is a functional attaining its norm on the vector $x_1 + \dots + x_k$ (assuming $x_{k+1}$ is an arbitrary point of $S$ if $x_1 + \dots + x_k = 0$).
\end{enumerate}

The sequence $\{x_i\}_1^\infty$ is well-defined. There exists $q \in S$ such that $\iprod{q}{p} \leq 0$ for any functional $p$, since the origin is in the convex hull of $S$.

 We will use induction on $k$ to show that $\norm{a_k} \leq R_k.$ Clearly,
$\norm{a_1} \leq R_1 = 1$. Assume $\norm{a_j} \leq R_j$ for all $j \in [k].$

For the induction step from $k$ to $k+1$, the key inequality in estimating the norm is the one from \Href{Corollary}{cor:max_deviation_via_mglx}:
\[
\norm{a_{k+1}} = \norm{a_k + x_{k+1}} \leq \norm{a_k} \zetap{\frac{\norm{x_{k+1}}}{\norm{a_k}}} \leq  \norm{a_k} \zetap{\frac{1}{\norm{a_k}}}.
\]
If $\norm{a_k} \leq R_{k+1} - 1$, then 
\[
\norm{a_{k+1}} \leq \norm{a_k} \zetap{\frac{1}{\norm{a_k}}}
 \stackrel{(\ast)}{\leq}
 \norm{a_k} \parenth{1 + \frac{1}{\norm{a_k}}} = \norm{a_k} + 1 \leq R_{k+1},
\]
where in $(\ast)$ we use the trivial inequality 
$\zetap{\tau} \leq 1 + \tau.$
If $\norm{a_k} \geq R_{k+1} - 1$, then 
\[
\norm{a_{k+1}} \leq \norm{a_k} \zetap{\frac{1}{\norm{a_k}}}
 \stackrel{(\ast\ast)}{\leq}
 \norm{a_k} \zetap{\frac{1}{R_{k+1} - 1}} \leq R_k \zetap{\frac{1}{R_{k+1} - 1}}  = R_{k+1},
\]
where in $(\ast\ast)$ we use the convexity of the function $\zetap{\cdot}$, which in turn follows from \Href{Lemma}{lem:monotonicity_zeta}.
\end{proof}

We claim without proof that in the most interesting case, when the modulus of smoothness satisfies the inequality $\mglx{\tau} \leq C \tau^p$ for some $p \in (1,2]$, the sequence $R_k$ is of the same order of magnitude as the sequence given in \eqref{eq:R_k_mglx}.

{\bf Acknowledgements.}  
I would like to thank Alexander Polyanskii, who motivated me to return to no-dimensional results.

\bibliographystyle{alpha}
\bibliography{../../work_current/uvolit}

\end{document}